\documentclass[a4paper,11pt]{amsart}          
\usepackage{amsfonts,amsmath,latexsym,amssymb} 
\usepackage{amsthm}      
\usepackage[dvipsnames]{xcolor}          

\usepackage{float}
\usepackage{tikz}
\usepackage{circuitikz}
\usepackage{hyperref}
\usepackage{mathtools}
\usepackage{enumitem}
\usepackage[]{mdframed}
\usepackage{algorithm2e}

\allowdisplaybreaks

\newtheorem{theorem}{Theorem}       
\numberwithin{theorem}{section}           
\newtheorem{lemma}[theorem]{Lemma}    
\newtheorem{proposition}[theorem]{Proposition}
\newtheorem{corollary}[theorem]{Corollary}

\theoremstyle{definition}

\newtheorem{definition}[theorem]{Definition}

\newtheorem{remark}[theorem]{Remark}

\DeclareMathOperator\supp{supp}

\newcommand{\GCD}{\operatorname{GCD}}
\newcommand{\wkto}{\xrightarrow{wk}}

\newcommand{\Q}{\mathbb{Q}}

\newcommand{\pr}{\pi}

\newcommand{\ev}{\nu}

\newcommand{\R}{\mathcal{R}}
\newcommand{\C}{\mathcal{C}}
\newcommand{\x}{\mathbf x}
\newcommand{\w}{\mathbf w}
\newcommand{\bfxi}{\mbox{\boldmath $\xi$}}
\newcommand{\norm}[1]{\left\lVert#1\right\rVert}
\newcommand{\abs}[1]{\left\lvert#1\right\rvert}

\author[R. Ebker, A. Muranova, M. Schmidt]{Ragon Ebker, Anna Muranova*, Max Schmidt}

\address{Ragon Ebker: Fernuniversität Hagen, Fakultät für Mathematik und Informatik, Universitätsstraße 47, 58097 Hagen, Germany}
\email{ragonebker@gmail.com}

\address{Anna Muranova: Uniwersytet Warmińsko-Mazurski w Olsztynie, 
Wydział Matematyki i Informatyki,
ul. Słoneczna 54, 
10-710 Olsztyn, Poland} 

\email{anna.muranova@matman.uwm.edu.pl}

\address{Max Schmidt: Technische Universität Dresden, Fakultät Mathematik, Zellescher Weg 12-14, 01069 Dresden, Germany}
\email{max.schmidt1@tu-dresden.de}
\title[Power iteration for power series matrices]{Power iteration for matrices with power series entries}
\thanks{*Corresponding author.}
\thanks{The authors acknowledge support by COST Action CA18232: Mathematical models for interacting dynamics on networks, since this project was started within 26th Internet Seminar ``Graphs and Discrete Dirichlet Spaces''.}

\begin{document}

\begin{abstract}

We prove the method of power iteration for matrices with at most finite entries from the Levi-Civita field $\C$ under the assumption that there exists an eigenvalue with the strictly largest in absolute value complex part. In this case the weak convergence of a start vector to the eigenvector, that corresponds to the largest eigenvalue, is proven. Further, we prove that the Rayleigh quotient of the largest eigenvector also converges weakly to the corresponding eigenvalue. As a corollary, the same holds for matrices and polynomials over the Puiseux series field. In addition to that, we deliver an implementation of our method in Python.

\smallskip
MSC2020:  65F15, 13F25, 12J15.
\smallskip

Keywords: eigenvalue, matrix with polynomial entries, power iteration, Levi-Civita field, Puiseux series, power series.
\end{abstract}

\maketitle


\section{Introduction}

Matrices with power series and polynomial entries become interesting in algebraic geometry, see e.g. \cite{Cox}, \cite{Manocha} and have a physical meaning in Hamiltonian systems and in the theory of electrical networks, see e.g. \cite{Falb, Friedland, Mehl, Muranova1, Muranova2}. Moreover, the question of finding eigenvalues of diagonalizable matrices over some algebraically closed field is closely related to the question of finding the roots of polynomials over this field. On one hand, the eigenvalues of a matrix are roots of its characteristic polynomial, on the other hand, roots of a polynomial can be found as eigenvalues of its companion matrix. 

 Until now it seems like there exists only very little research on the problem of solving matrices over the field of power series, as opposed to, for example, the research on solving polynomials.
In \cite{kitamoto1994approximate} the problem of finding eigenvalues is solved for matrices with polynomial entries, but very costly, via calculation of the roots of the characteristic polynomial. Applications of this solution can be found for example in \cite{phdthesis}, \cite{mori} and \cite{nano}. Further, in \cite{gentleman} and \cite{sun} the algorithms for finding determinants of matrices with polynomial entries are presented and analyzed. The classical approach to the problem of finding roots of polynomials over the field of power series consists of finding parts of the solution iteratively, lifting the solutions, and applying Newton iteration. Advanced versions of this method are discussed, for example, in \cite{neiger2017fast, poteaux2015improving}.

We present a power method algorithm for matrices with power series entries and believe that, besides being interesting on its own, it also opens the door to the development of further numerical approaches to finding the roots of multivariate polyomials, that might yield better stability and efficiency than existing methods. Our power method is proven to converge over the Levi-Civita field $\C$ with complex coefficients, which includes the Puiseux series field, and the polynomial ring, see Figure~\ref{figure_1}.

\bigskip
 \begin{figure}[H]
 \label{figure_1}
\begin{tikzpicture}
    \draw (3,0) ellipse (2cm and 1cm);
    \draw (1.5,0) ellipse (4cm and 1.5cm);
    \draw (0,0) ellipse (6cm and 2cm);

\node[draw=none,align=left] at (2.8,0) {Polynomial ring};
\node[draw=none,align=left] at (-0.7,0) {Puiseux series field};
\node[draw=none,align=left] at (-4.2,0) {Levi-Civita field};
\end{tikzpicture}
\caption{}
\end{figure}

The Levi-Civita field $\R$ is the smallest non-Archimedean real closed ordered field, which is Cauchy complete in the order topology, while the Levi-Civita field $\C$ is its algebraic closure. Moreover, $\R$ is the smallest Cauchy complete real closed field, containing a field of rational functions on one variable. The Levi-Civita field $\R$ is, in fact, a Cauchy completion of the field of Puiseux series, and the latest is widely studied in algebraic geometry.

In this paper, we show, that the power iteration algorithm for the  largest in modulus eigenvalue converges weakly for diagonalizable matrix over the Levi-Civita field $\C$ if the eigenvalue has the strictly larger in absolute value complex part, then other eigenvalues, Theorem \ref{thm::generalPM}.  This is an analogue of the classical power iteration method for complex numbers \cite{golub2013matrix}.  Moreover, as a simple corollary we show, that the power iteration algorithm for diagonalizable matrices over the Puiseux series converges coefficient-wise under the same conditions, Corollary \ref{cor:Puiseux}. 

The paper is organized as follows. In Section 2 we present the Levi-Civita fields and preliminary results. In Section 3 we prove auxiliary results on weak  convergence over the Levi-Civita fields, which up to our knowledge are not described in the literature, although some proofs are quite straightforward. In Section 4 we show that the power iteration algorithm converges for at most finite matrices, constructing the theoretical base for our method. In Section 5 we discuss, how the power iteration method can be applied to a wider range of matrices, in particular to the matrices over Puiseux series. In Section 6 we present examples, implemented in Python, and  experimental results.

\section{Preliminaries}
The Levi-Civita fields were introduced by Tullio Levi-Civita in 1862 \cite{LeviCivita}  and rediscovered by Ostrowski \cite{Ostrowski1935}, Neder \cite{Neder} and Laugwitz \cite{Laugwitz} independently later. The modern approach can be found in  \cite{Laugwitz1975a} and in \cite{Lightstone1975}. Analysis over the Levi-Civita fields has recent developments in the works by Berz, Shamseddine, and some others \cite{Berz, BOTTAZZI, SBAnalytical, ShamseddineBerz, Shamseddinethesis}.
Firstly, we recall the definition and main notations for the Levi-Civita fields $\R$ and $\C$.

\begin{definition}
A subset $Q$ of the rational numbers $\Bbb Q$ is called \emph{left-finite} if for every number $r\in \Bbb Q$ there are only finitely many elements of $Q$ that are smaller than $r$. 
\end{definition}

\begin{definition}
The set $\R$  (resp. $\C$) is defined as a set of functions from $\Q$ to $\mathbb R$ (resp. $\mathbb C$) with left-finite support.
\end{definition}
Note that any element $a\in \R$  (resp. $\C$) is uniquely determined by an ascending (finite or infinite) sequence $(q_n)$ of support points and a corresponding sequence $(a[q_n])$ of function values. We refer to the pair of the sequences $((q_n), (a[q_n]))$ as \emph{the table of $a$} \cite[Definition 3]{Berz}.

We  use the following terminology (all from the papers by Shamseddine and Berz, e.g.  \cite{Berz, SBAnalytical, ShamseddineBerz, Shamseddinethesis}):
\begin{definition}
For every $a$ in $\R$ (resp. $\C$) we denote:
\begin{itemize}
\item
$\supp (a)=\{q\in \Bbb Q\;\mid\;a[q]\ne 0\}$,
\item
$\lambda (a)=\min(\supp(a))$.
\end{itemize}
Note that $\lambda(a)$ is always well defined due to the left finiteness of the support.
Further, comparing any $a,b$ in $\R$, we write: 
\begin{itemize}
\item
$a\sim b$, if $\lambda(a)=\lambda(b)$,  
\item
$a\approx b$, if $\lambda(a)=\lambda(b)$ and $a[\lambda(a)]=b[\lambda(b)]$,
\item
$a=_r b$, if $a[q]=b[q]$ for all $q\le r\in \Bbb Q$.
\end{itemize}
 \end{definition}

\begin{lemma}\cite[Lemma 2.3]{Berz}
The relations $\sim$, $\approx$ and $=_r$ are equivalence relations.
\end{lemma}

Addition on $\R$  (resp. $\C$) is defined component-wise, i.e. for any $a,b\in \R$ (resp. $\C$) we define
$$
(a+b)[q]=a[q]+b[q].
$$
Multiplication is defined as follows:
$$
(a\cdot b)[q]=\sum_{\substack{q_a, q_b\in \Bbb Q,\\q_a+q_b=q}}a[q_a]\cdot b[q_b].
$$
From the definition of multiplication immediately follows that for any $a,b\in \R$  (resp. $\C$) holds
$$
\lambda(ab)=\lambda(a)+\lambda(b).
$$

 There exists a natural embedding of real numbers (resp. complex numbers) into $\R$  (resp. $\C$) defined by 
$\Pi:\mathbb R\to \R$ (resp. $\Pi:\mathbb C\to \C$):
$$
\Pi(s)[q]=
\begin{cases}
s, \mbox{ if }q=0,\\
0, \mbox{ otherwise.}
\end{cases}
$$
We immediately see that $\Pi$ is injective and direct calculations show that $\Pi(s_1+s_2)=\Pi(s_1)+\Pi(s_2)$ and $\Pi(s_1s_2)=\Pi(s_1)\cdot\Pi(s_2)$. By slight abuse of notations, we will refer to $\Pi(s)$ as $s$. It is clear, that the embedding preserves the algebraic structure and is injective, but not surjective.\\

\begin{theorem}\cite[Theorem 2.3 and Theorem 2.4]{Berz}\label{thm:2.5}
$(\R,+,\cdot, 0, 1)$ is a real closed field and  $(\C,+,\cdot, 0, 1)$ is an algebraically closed field.
\end{theorem}
Clearly, $\R$ can be embedded in $\C$ with the preservation of the algebraic structure.

The order on $\R$ is defined as follows: for any $a,b\in \R$ we write $a\succ b$ if $(x-y)[\lambda(a-b)]>0$. Then we denote the set of positive elements by 
$$
\R^+=\{a\in \R\;\mid\; a\succ 0\}.
$$
We use the notation $a\succeq 0$, if $a\in \R^+\cup\{0\}$. Then the following can be proven directly by Theorem \ref{thm:2.5} and definitions:

\begin{theorem}
$(\R,+,\cdot, 0, 1, \succ)$ is an ordered field and $(\C,+,\cdot, 0, 1)$ is its algebraic closure.
\end{theorem}

Now we can define an analogue of absolute value on $\R$ as $|\cdot|:\R\to \mathcal \R^+\cup\{0\}$: for any $a\in \R$ we define
$$
|a|=\begin{cases}
a, \mbox{ if }a\succeq 0,\\
-a, \mbox{ otherwise}.
\end{cases}
$$
 Moreover, the corresponding absolute value is
defined on $\C$ in the natural way \cite[Definition 11]{Berz}: since for $z\in \C$ there exists a unique representation $z=a+ib, a,b\in \R$ we define 
$$
|z|:=\sqrt{a^2+b^2}
$$
and, further, we denote by
$$
\overline z:=a-ib.
$$
the conjugated of $z$. Similar to the case of complex numbers we have 
$$
\overline z z=|z|^2,
$$ 
for any $z\in \C$.

It is straightforward to prove, that $|\cdot|$ on the Levi-Civita fields is an analogue of the norm, i.e. it is positive, satisfies a triangle inequality, and $|a|=0$ if and only if $a=0$.

The number $d\in \R$ is defined as follows:
$$
d[q]=\begin{cases}
1, q=1,\\
0, \mbox{ otherwise.}
\end{cases}
$$
It is easy to see, that $d$ is smaller than any real number (as an element of $\mathcal R$), i.e. $d$ plays a role of infinitesimal in $\R$ and, therefore, the Levi-Civita field $\R$ is non-Archimedean. Nevertheless, $d$ is not the only infinitesimal in $\R$ (e.g. $d^2$, $5d$, $d+d^3$ are infinitesimals).

Therefore, for $a\in \R$ (resp. $a\in \C$) we say that:
\begin{itemize}
\item
$a$ is \emph{infinitely small} if $\lambda(a)>0$,
\item
$a$ is \emph{infinitely large} if $\lambda(a)< 0$.
\end{itemize}
 Moreover, we say  that
$a$ is \emph{at most finite} if $\lambda(a)\ge 0$.

Now we describe two types of convergence over Levi-Civita fields: the strong one (or convergence with respect to the order topology) and the weak one. The latest one is of special importance for us, since any element of $\mathcal R$ (resp. $\mathcal C$) can be presented as formal power series, see Proposition \ref{prop::powerseriesrepresentation} over real (resp. complex numbers),  and weak convergence ensures convergence of the coefficients in the usual sense. The definitions and main properties of both convergences can be found for $\R$ in e.g. \cite{Berz, SBAnalytical, Shamseddinethesis} and for $\C$ in \cite{SBAnalytical}. We present them here for the sake of completeness.

\begin{definition}
We call the sequence $(a_n)$ in $\R$ (resp. $\C$) strongly convergent to the limit $a$ in $\R$ (resp. $\C$), if for any $\epsilon \in \R^+$ there exists  $N_0\in \mathbb N$  such that for all $n\ge N_0$ we have
$$
|a_n-a|\prec \epsilon.
$$
In this case, we use the notation $\lim_{n\to \infty} a_n=a$ or $a_n\to a$.
\end{definition}

The following Lemma follows  from the fact, that strong convergence is convergence with respect to the order topology \cite[Definition 12 and Definition 16]{Berz}:
\begin{lemma}
 Let $(a_n)$ be a sequence in $\R$ (resp. in $\C$) that is strongly convergent. Then its limit is unique.
\end{lemma}

\begin{definition}
The sequence $(a_n)$  in $\R$ (resp. $\C$)   is a \emph{strongly Cauchy sequence} if  for any $\epsilon \in \R^+$ there exists  $N_0\in \mathbb N$  such that for all $n, m\ge N_0$ we have
$$
|a_n-a_m|\prec \epsilon.
$$ 
\end{definition}

\begin{theorem}\cite[Theorem 4.5]{Berz}
The Levi-Civita fields $\R$ (resp. $\C$) is Cauchy complete with respect to the strong convergence, i.e. $(a_n)$  is a strongly Cauchy sequence in $\R$ (resp. $\C$) if and only if $(a_n)$ converges strongly.
\end{theorem}

\begin{proposition}\cite[Theorem 4.3]{Berz}\label{prop::powerseriesrepresentation}
Let $((q_i), (a[q_i]))$ be the table of $a\in \R$ (resp.  $a\in \C$). Then the sequence $a_n=\sum_{i=0}^n a[q_i]\cdot d^{q_i}$, converges strongly to the limit $a$. Hence we can write
\begin{equation}\label{eq::powerseriesrepresentation}
a=\sum_{i=0}^\infty a[q_i]\cdot d^{q_i}.
\end{equation}
\end{proposition}
Note, that there is a slight abuse of notations in \eqref{eq::powerseriesrepresentation}, since the sum can be also finite. Then $a_{N_0}=a$ for some $N_0$ and we put $a_n=a_{N_0}$ for all $n\ge N_0.$

To define a weak convergence, we first need to define the following semi-norm:

\begin{definition}
Given $r\in \Bbb Q$, we define a map $\|\cdot\|_r:\R \mbox{ (resp. }\C\mbox{)} \to\R$ as follows:
$$
\|a\|_r=\sup_{q_i\le r}|a[q_i]|.
$$
\end{definition}

The fact that $\|\cdot\|_r$ is an analogue of semi-norm for all $r\in \Bbb Q$ is straightforward to prove, more details one can find in e.g. \cite{Berz, SBAnalytical, Shamseddinethesis}.

\begin{definition}\label{def:weakconvergence}
A sequence $(a_n)$ in $\R$ (resp. $\C$) is said to be {\em weakly convergent} if
there exists an $a$ in $\R$ (resp. $\C$) such that for all $r>0, r\in \Bbb R$, there exists $N_0\in \Bbb N$ such that
$$
\|a_n-a\|_{1/r}<r\mbox{ for all }n>N_0.
$$
If that is the case, we call $a$ the {\em  weak limit} of the sequence $(a_n)$, and we write
$$
a=wk\lim_{n\to \infty}a_n \mbox{ or }a_n\wkto a.
$$
\end{definition}

The proof of the following Lemma is straightforward and for the Levi-Civita field $\R$ is presented in \cite[Lemma 4.10]{Shamseddinethesis}.
\begin{lemma}
 Let $(a_n)$ be a sequence in $\R$ (resp. $\C$)  that is weakly convergent. Then $(a_n)$ has
exactly one weak limit in $\R$  (resp. $\C$).
\end{lemma}

The relation between weak and strong convergences is given by the following:
\begin{lemma}\cite[Theorem 4.7]{Berz}
Strong convergence in $\R$ (resp. $\C$) implies weak convergence to the same limit.
\end{lemma}

Further, there is an analogue of the Cauchy sequence with respect to weak topology:

\begin{definition}
The sequence $(a_n)$  in $\R$ (resp. $\C$)   is called \emph{weakly Cauchy sequence} if  for any $r>0$ in $\mathbb R$ there exists  $N_0\in \mathbb N$  such that for all $n, m\ge N_0$ we have
$$
\|a_n-a_m\|_{1/r}< r.
$$ 
\end{definition}

\begin{lemma}\label{lem::wkCauchy}
Every weakly convergent $\R$ (resp. $\C$) sequence is weakly Cauchy.
\end{lemma}

Note that $\R$ and $\C$ are not Cauchy complete with respect to the weak convergence. For further details see e.g. \cite{Berz, Shamseddinethesis}.

The definition of weak convergence is complicated to understand and use, therefore in proofs, we will use Theorem \ref{thm:weakconvergence}, formulated below, which provides us with a criterion for weak convergence of a sequence in terms of coefficients and supports of its elements. To formulate Theorem \ref{thm:weakconvergence} we need the following definition of regularity.
\begin{definition}
A sequence $(a_n)$ in $\R$ (resp. $\C$)  is called {\em regular} if and only if
the union of the supports of all elements of the sequence, i.e. the set $\cup_{n\in\Bbb N} \supp a_n$, is a left-finite set.
\end{definition}

\begin{theorem}\cite[Convergence Criterion for Weak Convergence]{SBAnalytical}
\label{thm:weakconvergence}
\hspace{0pt}
\begin{enumerate}[label=$\arabic*.$]
\item
Let the sequence $(a_n)$ in $\R$ (resp. $\C$) converges weakly to the limit $a$. Then, the sequence $(a_n[q])$ converges to $a[q]$ in $\Bbb R$ (resp. $\mathbb C$),
for all $q\in \Bbb Q$, and the convergence is uniform on every subset of $\Bbb Q$ bounded above.
\item
Let on the other hand $(a_n)$ in $\R$ (resp. $\C$) be regular, and let the sequence $(a_n[q])$ converge in $\Bbb R$ (resp. $\Bbb C$) to
$a[q]$ for all $q\in \Bbb Q$. Then $a_n$ converges weakly in $\R$ (resp. $\C$) to $a$.
\end{enumerate}
\end{theorem}

The following several Lemmas describe more properties of weak convergence.
\begin{lemma}\cite[Lemma 4.3]{Berz}\label{lemma:reg}
Let $(a_n)$ and $(b_n)$ be regular sequences in
$\R$ (resp. $\C$). Then the sequence of the sums, the sequence of the products, any rearrangement,
as well as any subsequence of one of the sequences are regular.
\end{lemma}

\begin{lemma}\cite[Lemma 4.4]{Berz}\label{lemma:strreg}
Every strongly convergent sequence is regular.
\end{lemma}

\begin{lemma}\label{lemma::continuityofwc}
Let $(a_n)$ and $(b_n)$ be two sequences in $\R$ (resp. $\C$) converging weakly to $a$ and $b$, respectively. Then $a_n+b_n\wkto a+b$. Moreover, if $(a_n)$ and $(b_n)$ are both regular, then  $a_nb_n\wkto ab$.
\end{lemma}

\begin{proof}
The proof for $\R$ is presented in \cite[Lemma 4.17 and Theorem 4.8]{Shamseddinethesis}, while the proof for $\C$ follows exactly the same outline.
\end{proof}

\begin{lemma}\cite[Corollary 3.2]{SBAnalytical}\label{lem:powerreg}
The sequence $(a^k)$ is regular for $a$ in $\R$ (resp. $\C$) if and only if $\lambda(a)\ge 0$.
\end{lemma}

The following Theorem can be considered as a main ingredient of the proof of power iteration method:

\begin{theorem}\cite[Power Series with Purely Real or Complex Coefficients]{SBAnalytical}\label{thm::powerseries}

Let $\sum_{k=0}^{\infty}s_k X^k$be a power series with purely real (resp. complex) coefficients and with the classical radius of convergence equal to $\eta$. Let $x$ in $\R$ (resp. $\C$), and let $a_n(x)=\sum_{k=0}^{n}s_k x^k$  in $\R$ (resp. $\C$). Then, for $|x|<\eta$ and $|x|\not\approx \eta $ the sequence $(a_n(x))$ converges weakly.
\end{theorem}

\section{Properties of weak convergence}
 In this section, we prove auxiliary results on weak  convergence over the Levi-Civita fields. The most interesting results here are Theorem \ref{thm::regularityOfinverse} and Theorem \ref{thm::squareroot}, which show us the continuity of weak convergence with respect to taking inverses and roots in some cases. These theorems are crucial in the proof of power iteration algorithm.

Firstly, we formulate an equivalent definition of weak convergence.
\begin{proposition}\label{prop:weakconvergence}
Let $(a_n)$ in $\R$ (resp. $\C$) be a sequence. Then $a=wk\lim_{n\to \infty}a_n$ if and only if
there exists $a\in \R$ (resp. $\C$) such that for all $r_0,r_1>0$ in $\Bbb R$, there exists $N_0\in \Bbb N$ such that
$$
\|a_n-a\|_{r_0}<r_1\mbox{ for all }n>N_0.
$$

\end{proposition}
\begin{proof}
The ``if'' direction is clear by taking $r_0=1/r, r_1=r.$

For the ``only if direction'' let us consider two cases. If $r_1 r_0<1$ we can estimate
$$
\|a_n-a\|_{r_0}=\sup_{q_i\le r_0} |a_n[q_i]-a[q_i]|\le \sup_{q_i\le 1/r_1} |a_n[q_i]-a[q_i]|=\|a_n-a\|_ {1/r_1}<r_1.
$$
If  $r_1 r_0\ge 1$ we can argue in the following way:
$$
\|a_n-a\|_{r_0}<1/r_0\le r_1.
$$
\end{proof}

	\begin{lemma}
    \label{lemma::cauchy}
		For a weakly convergent series $\sum_{k=0}^{\infty} a_k$  in $\R$ (resp. $\C$) the sequence $(a_k)$ converges weakly to zero.
	\end{lemma}
	\begin{proof}
		Let $s_n:=\sum_{k=0}^{n} a_k$ and $s=wk\lim_{n\to \infty}a_n$ be the limit. By Proposition \ref{prop:weakconvergence}, for any $r_0, r_1>0$ in $\mathbb R$ there is a $N\in \mathbb N$ such that $\norm{s_n-s}_{r_0}<r_1/2$ for all $n\geq N$. Therefore, for any $n \ge N$ we have,  
		\begin{align*}
			\norm{s_n-s_{n+1}}_{r_0}=\norm{s_n-s+s-s_{n+1}}_{r_0}\leq\norm{s_n-s}_{r_0}+\norm{s_{n+1}-s}_{r_0}\leq r_1.
		\end{align*}
	\end{proof}

\begin{lemma}\label{lemma::convofpowers}
Let $a\in \R$ (resp. $\C$). Then for the sequence $(a^k)_{k\in \Bbb N}$ the following assertions hold:
\begin{enumerate}
\item
If  $\lambda(a)>0$, then $wk\lim_{k\to \infty}a^k=0$. Moreover, $(a^k)$ is regular.
\item
If  $\lambda(a)=0$ and $\left|a[0]\right|<1$, then $wk\lim_{k\to \infty}a^k=0$. Moreover, $(a^k)$ is regular.
\end{enumerate}

\end{lemma}
\begin{proof}
\begin{enumerate}
\item
We see that $a^k$ converges strongly to zero, since $\lambda(a^k)=k\lambda(a)\to+\infty$ and the regularity of $(a^k)$ follows from Lemma \ref{lemma:strreg}.
\item
Let $q_0=\lambda(a)=0$. Applying Theorem \ref{thm::powerseries}, we get that the series $\sum_{k=0}^\infty a^k$ is absolutely weakly convergent since it is a geometric series for real $a$. Therefore, by Lemma \ref{lemma::cauchy}, $a^k\wkto 0$.
Further, the sequence $(a^k)$ is regular by Lemma \ref{lem:powerreg}.
\end{enumerate}
\end{proof}

Now we prove Theorem \ref{thm::regularityOfinverse} and Theorem \ref{thm::squareroot}, which are crucial for the proof of convergence of power iteration algorithm, presented in Section~4.

\begin{theorem}\label{thm::regularityOfinverse}
Let $(a_k)_{k\in \Bbb N}\subset \Bbb \R$ and $a_k \wkto a$ as $k\to \infty$. Let $ a[0]\ne 0$ and the following conditions holds:
\begin{itemize}
\item
all $(a_k)_{k\in \Bbb N}$ and $a$ are at most finite,
\item
$(a_k)_{k\in \Bbb N}$ is regular.
\end{itemize}
Then 
$$
\dfrac{1}{a_k}\wkto \dfrac{1}{a}.  
$$
\end{theorem}
\begin{proof}
Firstly we prove, that the sequence $(b_k)_{k\in \Bbb N}$, where $b_k=\dfrac{1}{a_k}$, is regular. Let $\bar q\in \mathbb Q$ be fixed. To prove that $(b_k)$ is regular, we need to prove that the set 
$$
\{q\le\bar q\;\mid\; q\in \supp(b_k),k\in \mathbb N\}
$$
is finite. Let us write $a_k$ as
$$
a_k=P_k+\alpha_k,
$$
where $\supp P_k\subset(-\infty,\bar q]$ and $\supp \alpha_k \subset (\bar q,+\infty)$.  Further,
$$
b_k=\dfrac{1}{a_k}=\dfrac{1}{P_k+\alpha_k}
$$
and, consequently,
\begin{equation}\label{eq::bkPk}
b_k+\dfrac{b_k\alpha_k}{P_k}=\dfrac{1}{P_k}.
\end{equation}
Since $P_k[0]=a_k[0]\to a[0]$ by Theorem \ref{thm:weakconvergence}  we can assume w.l.o.g that $P_k[0]\ne 0$. Moreover, $P_k$ is at most finite. Therefore, $\dfrac{1}{P_k}$ coincides with the Maclaurin series of $P_k$ in terms of variable $d^{\frac1p}$,
where
$$
p={\GCD\left(\{n\;\mid\;\frac{m}{n}\in \bigcup_{k\in \mathbb N}\supp(P_k)\mbox{ is irreducible}\}\right)},
$$
exists, since $\cup_{k\in \Bbb N}\supp(P_k)$ is finite, due to the regularity of sequence $(a_k)$.
Thus, $\left(\dfrac{1}{P_k}\right)_k$ is regular. Further, note that for each $k$ we have
\begin{align*}
\min \left(\supp\dfrac{b_k\alpha_k}{P_k}\right)&= \min (\supp {b_k})+ \min (\supp {\alpha_k})+ \min \left(\supp \frac {1}{P_k}\right)\\&>0+\bar q+0=\bar q.
\end{align*}
Therefore, due to \eqref{eq::bkPk}
$$
(\supp b_k)\cap (-\infty,\bar q]=\left(\supp \frac {1}{P_k}\right)\cap (-\infty,\bar q]
$$
and 
$$
\{q\le\bar q\;\mid\; q\in \supp(b_k),k\in \mathbb N\}=\{q\le\bar q\;\mid\; q\in \supp\left(\frac {1}{P_k}\right),k\in \mathbb N\}
$$
is finite, since $\left(\dfrac {1}{P_k}\right)_k$ is regular. This yields the regularity of $b_k=\dfrac{1}{a_k}$.

Further, in order to prove $b_k\wkto \dfrac1a$, let us denote 
$$
c_k:=b_k-\frac{1}{a}=\frac{1}{a_k}-\frac{1}{a}.$$

Note that $c_k$ is at most finite. Further, $(c_k)_{k\in \Bbb N}$ is regular, since ${b_k}$ is regular and
\begin{equation}\label{akck}
a_kc_k =1-a_k\dfrac{1}{a}\wkto 0,
\end{equation}
since $a_k\wkto a$ and weak convergence is continuous with respect to multiplication for the regular sequences by Lemma \ref{lemma::continuityofwc}.

Since $c_k$ is regular and at most finite, $Q:=\cup_{k\in \Bbb N} \supp (c_k)$ is left-finite and $\min(Q)\ge 0$, hence we can write
$Q=(q_i)_{i=0}^\infty$, where $q_0\ge 0$ and $(q_i)$ is strictly increasing.

Further, let us denote 
$$
C_{0,k}=c_k
$$
and for $i={0,1, 2,\dots}$ we put
$$
C_{i+1,k}:=C_{i,k}-C_{i,k}[q_i]d^{q_i}=C_{i,k}-c_k[q_i]d^{q_i}=\sum_{j=i+1}^{\infty}c_k[q_j]d^{q_j}.
$$ 
Firstly, we prove by induction in $i$, that 
\begin{equation}\label{eq::akCik}
a_kC_{i,k}\wkto 0.
\end{equation}
Indeed, for $i=0$ equation \eqref{akck} leads to
$$
a_kC_{0,k}=a_k c_k \wkto 0.
$$
Assume \eqref{eq::akCik} holds for $i$. Note that $\lambda(C_{i,k})\ge q_{i}$, hence,
 $$
a_k[0] c_k[q_i]=a_k[0]C_{i,k}[q_i]=(a_kC_{i,k})[q_i]\to 0 \mbox{ in }\Bbb R,
$$ 
and, therefore, 
\begin{equation}\label{cqi}
c_k[q_i]\to 0  \mbox{ in }\Bbb R,
\end{equation}
since $a_k[0]\to b[0]\ne 0$ in $\Bbb R$. Hence, $c_k[q_i]d^{q_i}\wkto 0 $.
Further, we have
$$
a_kC_{i+1,k}=a_kC_{i,k}-a_kc_k[q_i]d^{q_i}\wkto 0,
$$
since both terms converge weakly to $0$.

Altogether, \eqref{eq::akCik} holds for all $i$, and we already have shown, see  \eqref{cqi}, that from there follows $c_k[q_i]\to 0$ in $\Bbb R$. Finally, since 
$c_k$ is regular and  $c_k[q_i]\to 0$ in $\Bbb R$ for all $i\in \Bbb N$, we obtain by convergence criterium, Theorem \ref{thm:weakconvergence}, that 
$c_k\wkto 0,$
hence
$$
\dfrac{1}{a_k}\wkto \dfrac{1}{a}.
$$
\end{proof}

To prove Theorem \ref{thm::squareroot}, where we show the convergence of  square roots of a sequence in some cases, we first need the following:
\begin{lemma}\label{lemma::sqrtOfPoly}
Let $(P_k)_k\subset\R, P\in \R$ can be represented as polynomials of degree at most $n$ of $d^\frac{1}{p}$, for some rational $p>0$, $P[0]\ne 0$ and
$$
P_k\wkto P.
$$
Then for any $r\in \mathbb Q^+$ with $r\le 1$, we have $P_k^r\wkto P^r$.
\end{lemma}

\begin{proof}
Let us denote $x:=d^\frac{1}{p}$ and
let $P_k=\sum_{i=0}^n a_k[i/p]x^i$, $P=\sum_{i=0}^n a[i/p]x^i$. W.l.o.g by convergence criterium for weak convergence, Theorem \ref{thm:weakconvergence}, we can assume, that $a_k[0]\ne 0$ for all $k$. Then $P^r_k$  and $P^r$  can be calculated as Maclaurin series of $P^r(x)$, with the substitution $x=d^\frac{1}{p}$, i.e.
\begin{equation}\label{eq::sqrtPk}
 P^r_k=\sum_{n=0}^\infty \left.\dfrac{({P^r_k(x)})^{(n)}_x}{n!}\right|_{x=0}d^\frac{n}{p},
\end{equation}
and
$$
 P^r=\sum_{n=0}^\infty \left.\dfrac{({P^r(x)})^{(n)}_x}{n!}\right|_{x=0}d^\frac{n}{p},
$$
It is clear due to \eqref{eq::sqrtPk} that the sequence $\left(
P_k^r\right)_k$ is regular.  Moreover,
$$
\left.\dfrac{( {P^r_k(x)})^{(n)}_x}{n!}\right|_{x=0}\to \left.\dfrac{( {P^r(x)})^{(n)}_x}{n!}\right|_{x=0}
$$
in $\Bbb R$ due to Faà di Bruno's formula, applied to the composition of rational degree and polynomial (note, that  $P_k(0)\ne 0$). Therefore, 
$$
P^r_k\wkto P^r
$$
by the criterium for the weak convergence, Theorem \ref{thm:weakconvergence}.
\end{proof}

\begin{theorem}\label{thm::squareroot}
Let $(a_k)_{k\in \Bbb N}\subset \Bbb \R$ and $a_k \wkto a$ as $k\to \infty$. Let $ a[0]\ne 0$ and the following conditions holds:
\begin{itemize}
\item
all $(a_k)_{k\in \Bbb N}$ and $a$ are at most finite,
\item
$(a_k)_{k\in \Bbb N}$ is regular,
\end{itemize}
Then 
$$
\sqrt{a_k}\wkto \sqrt{a}.  
$$
\end{theorem}
\begin{proof}
Firstly, let us prove that the sequence $b_k:=\sqrt{a_k}$ is regular. We use the approach similar to the one in proof of Theorem \ref{thm::regularityOfinverse}. Let  $\bar q\in \mathbb Q$ be fixed. We write
\begin{equation}\label{eq::ak}
a_k=P_k+\alpha_k,
\end{equation}
where $\supp P_k\subset(-\infty,\bar q]$ and $\supp \alpha_k \subset (\bar q,+\infty)$. W.l.o.g we can assume that $P_k[0]\ne 0$, since $a[0]\ne 0$. Further, let
$$
b_k=R_k+\beta_k,
$$
where $\supp R_k\subset(-\infty,\bar q]$ and $\supp \beta_k \subset (\bar q,+\infty)$.
Then from $b_k^2=a_k$ we conclude that $b_k$ is at most finite, i.e. so is $R_k$, and $b_k[0]\ne 0$, moreover we get:
$$
R_k^2+2\beta_kR_k+\beta_k^2=P_k+\alpha_k,
$$
and, comparing the supports of the summands, we get
$$
R_k^2=_{\bar q}P_k.
$$
Then we can conclude, that $R_k=_{\bar q}\sqrt{P_k}$. Indeed, assume 
\begin{equation}\label{eq::star}
R_k=\sqrt{P_k}+S_k,
\end{equation} 
where $S_k$ is at most finite due to at most finiteness of $R_k$ and $\sqrt{P_k}$. Then 
$$
R_k^2=P_k+2S_k\sqrt P_k+S_k^2=_{\bar q}P_k.
$$
Therefore,
$$
2S_k\sqrt P_k+S_k^2=_{\bar q}0,
$$
and, hence,
\begin{equation}\label{eq::lambdap}
\lambda \left(S_k\right)+\lambda \left(2\sqrt P_k+S_k\right)=\lambda \left(S_k(2\sqrt P_k+S_k)\right)>\bar q.
\end{equation}
Note that both summands are at most finite and $\left (2\sqrt P_k\right)[0]\ne 0$ . Hence, for $\lambda_0:=\lambda\left(2\sqrt P_k+S_k\right)$ we have two options:
\begin{itemize}
\item
if $\left(2\sqrt P_k\right)[0]\ne -S_k[0]$, then $\lambda_0=0$. Therefore, $\lambda \left( S_k\right)>\bar q$, from where, due to \eqref{eq::star}, follows $R_k=_{\bar q}\sqrt{P_k}$.
\item
if $\left(2 \sqrt P_k\right)[0]=-S_k[0]$ and, hence, $\lambda(S_k)=\lambda(-2 \sqrt P_k)=0$, then let us put 
$$
\overline p:=\lambda \left(2\sqrt P_k+S_k\right),
$$ 
and we get
$$
0+\overline p=\lambda \left(S_k\right)+\lambda \left(2\sqrt P_k+S_k\right)>\bar q,
$$
where the last inequality follows from \eqref{eq::lambdap}, i.e. $\overline p>\bar q$, hence,  
$$
\lambda \left(2\sqrt P_k+S_k\right)>\overline q,
$$
 which leads to $S_k=_{\bar q}-2\sqrt P_k$, i.e. $R_k=_{\bar q}~-\sqrt P_k$, and that  is not possible, since $R_k\succ 0$ by definition of the square root in an ordered field.
\end{itemize}
Therefore, $R_k=_{\bar q} \sqrt P_k$. We know that $(a_k)$ is regular, hence, $\cup_{k\in \Bbb N}\supp(P_k)$ is finite by definition of $P_k$.  Therefore, the sequence of Maclaurin expansions of $(\sqrt P_k)_k$ of variable $d^{\frac1p}$,
where
$$
p={\GCD\left(\{n\;\mid\;\frac{m}{n}\in \bigcup_{k\in \mathbb N}\supp(P_k)\mbox{ is irreducible}\}\right)},
$$
exists and is regular.
Then from 
$$
b_k=_{\bar q} R_k=_{\bar q} \sqrt P_k,
$$ 
follows, that the set
$$
\{q\in \cup_{k\in \Bbb N}\supp b_k\;\mid\; q<\bar q\}
$$
is finite, i.e. $(b_k)_k$ is regular. 

Moreover, from the above, we have that for any $\bar q$ the following holds:
$$
b_k[\bar q]=R_k[\bar q]=(\sqrt  P_k)[\bar q]
$$
and due to Lemma \ref{lemma::sqrtOfPoly} we have $(\sqrt  P_k)[\bar q]\to \left(\sqrt P\right)[\bar q]$,
where $P=wk \lim_{k\to \infty}P_k$. Further, due to \eqref{eq::ak} and $a_k\wkto a$ we can write
$$
a=P+\alpha,
$$
where $\alpha=wk \lim_{k\to\infty} \alpha_k$, i.e. $\supp \alpha\in (\bar q,+\infty)$ and by the same argument as above for $a_k$, we get $\sqrt a=_{\bar q}\sqrt P$. Therefore,
$$
b_k[\bar q]=(\sqrt  P_k)[\bar q]\to (\sqrt a) [\bar q]
$$
in $\Bbb R$. Now due to an arbitrary choice of $\bar q$ and regularity of $b_k$ we can conclude by convergence criterium for weak convergence, Theorem \ref{thm:weakconvergence}, that 
$$
b_k\wkto \sqrt a.
$$
\end{proof}
\section{Power iteration algorithm}

In this section, we will prove the convergence of the power iteration algorithm in $\ell^2$-norm (Theorem \ref{Thm::pil2}) and in maximum norm (Theorem \ref{Thm::pimax}) under the certain natural conditions. These are the main results of this note.

In order to formulate the results, we introduce the definitions of the corresponding vector norms, which are similar to the classical case.

\begin{definition}\label{def::norms}
For any $\x=(\x^1,\dots, \x^n)^\mathrm{T} \in \R^n$ (resp. $\C^n$) a \emph{maximum norm} is defined as follows
$$
\|\x\|_{\max}=\max |\x^i|
$$
and an \emph{$\ell^2$-norm} is given by
$$
\|\x\|_2 =\sqrt{|\x^1|^2+\dots+|\x^n|^2}.
$$
\end{definition}
Note that for $\x\in \R^n$ or $\C^n$ we have $\|\x\|_{\max}, \|\x\|_2\in \R^+\cup\{0\}$. Moreover, $\|\x\|_{\max}=\|\x\|_2=0$ if and only if $\x=\mathbf 0:=(0, \dots, 0)$. It follows directly from the definitions, that the triangle inequality holds in both cases.
\begin{lemma}\label{lemma::xik}
Let a matrix $A\in \C^{n,n}$ be diagonalizable and its eigenvalues $\ev_1,\dots \ev_n$ are at most finite and satisfy the following condition 
\begin{equation}\label{eq::lemmacondev1}
1=\ev_1[0]>\left|\ev_2[0]\right|\ge\dots\ge|\ev_n[0]|.
\end{equation}
Further, let $\x\in \C^n$ be a vector, such that for the representation
\begin{equation}\label{eq::xaswj}
\x=\w_1+\w_2+\dots +\w_n,
\end{equation}
where $\w_j$ is an eigenvector of $\ev_j$ with at most finite entries holds:
\begin{enumerate}[label=(\roman*)]
\item
all $\w_j$ are at most finite,
\item
$\w_1[0]\ne \mathbf 0$,
\end{enumerate}

Then the following holds:
\begin{enumerate}[label=$\arabic*.$]
\item
 The sequence $\bfxi_k:=\sum_{j=2}^n\left(\frac{\ev_j}{\ev_1}\right)^k  \w_j$ is regular and
$$
\bfxi_k\wkto \mathbf 0.
$$
\item
$$
\dfrac{1}{\|\w_1+\bfxi_k\|_{2}}\wkto \dfrac{1}{\|\w_1\|_{2}}\
$$
\end{enumerate}

\bigskip

Moreover, if there exists a unique coordinate $i_0$ such that $
|\w^{i_0}_1[0]|=\max_i |\w^{i}_1[0]|$,
then
\begin{enumerate}[label=$2^*.$]
\item
$$
\dfrac{1}{\|\w_1+\bfxi_k\|_{\max}}\wkto \dfrac{1}{\|\w_1\|_{\max}}
$$
\end{enumerate}
\end{lemma}

\begin{proof}
\begin{enumerate}[label=$\arabic*.$]
\item
Since for any $j>1$ the ratio $\left(\dfrac{\ev_j}{\ev_1}\right)[0]$ is at most finite and either $\abs{\left(\dfrac{\ev_j}{\ev_1}\right)[0]}<1$ or $\left(\dfrac{\ev_j}{\ev_1}\right)[0]=0$ due to \eqref{eq::lemmacondev1}, the following holds 
$$
\left(\frac{\ev_j}{\ev_1}\right)^k \wkto \mathbf 0
$$
due to Lemma \ref{lemma::convofpowers}.
\item
holds due to a Definition \ref{def::norms} of $\ell^2$-norm, Lemma \ref{lemma::continuityofwc} and Theorems \ref{thm::regularityOfinverse} and \ref{thm::squareroot}.
\end{enumerate}
\begin{enumerate}[label=$2^*.$]
\item
Firstly, let us show, that for $k$ large enough
\begin{equation}\label{eq::normmaxi0}
\|\w_1+\bfxi_k\|_{\max}=|\w_1^{i_0}+\bfxi_k^{i_0}|,
\end{equation}
for  the unique $i_0$ such that $|\w_1^{i_0}[0]|=\max_i |\w_1^i[0]|$.
Indeed, for all $i$ we have
\begin{equation}\label{eq::brr1}
\left|\w_1^{i}[0]+\bfxi_k^{i}[0]\right|\le  \left|\w_1^{i}[0]\right|+\left|\bfxi_k^{i}[0]\right| 
< |\w_1^{i_0}[0]|-|\bfxi_k^{i_0}[0]|\le |\w_1^{i_0}[0]+\bfxi_k^{i_0}[0]|,
\end{equation}
for all $k$ such that 
\begin{equation}\label{eq::condK}
|\bfxi_k^{j}[0]|<\dfrac{\min_i \left(|\w_1^{i_0}[0]|-\left|\w_1^{i}[0]\right|\right)}{2}.
\end{equation}
for all $j$. Note that condition \eqref{eq::condK} holds for all $k$ large enough,  since $\bfxi_k\wkto \mathbf 0$ by (1) of the Theorem and, hence, $\bfxi_k[0]\to 0$ in $\Bbb R$  due to Theorem \ref{thm:weakconvergence}. Then \eqref{eq::normmaxi0} follows from \eqref{eq::brr1} due to the fact that $\w_1$ and  $\bfxi_k$ are at most finite.

Moreover, we have $\w_1^{i_0}+\bfxi_k^{i_0}\wkto \w_1^{i_0}$ due to $\bfxi_k\wkto \mathbf 0$, and, further,
$$
|\w_1^{i_0}+\bfxi_k^{i_0}|\wkto |\w_1^{i_0}|,
$$
follows immediately from the definition of weak convergence. Now
we can apply Theorem \ref{thm::regularityOfinverse} to \eqref{eq::normmaxi0} to get $(2^*)$.  
\end{enumerate}
\end{proof}

\begin{theorem}[Power iteration with $\ell^2$-norm]\label{Thm::pil2}
Let $A\in \C^{n,n}$ be a diagonalizable matrix  and its eigenvalues $\ev_1,\dots \ev_n\in \R$ are at most finite and satisfy the following condition 
\begin{equation}\label{cond::eigenvalues}
1=\ev_1[0]>\left|\ev_2[0]\right|\ge\dots\ge|\ev_n[0]|.
\end{equation}
Further, let $\x\in \C^n$ a vector, such that for the representation
\begin{equation}\label{eq::xaswj}
\x=\w_1+\w_2+\dots +\w_n,
\end{equation}
where $\w_j$ is an eigenvector of $\ev_j$, holds:
\begin{enumerate}[label=(\roman*)]
\item
all $\w_j$ are at most finite,
\item
$\w_1[0]\ne \mathbf 0$,

\end{enumerate}
Then for $\x_k:={A^k \x}$ holds
$$
\dfrac{\x_k}{\|\x_k\|_{2}} \wkto \dfrac{\w_1}{\|\w_1\|_{2}}.
$$
\end{theorem}

\begin{proof}
The representation \eqref{eq::xaswj} for $\x$ is unique, since $\ev_1, \dots,\ev_n$ is a basis of $\mathcal C^n$. Moreover, we have 
$$
\x_k=A^k\x=\sum_{j=1}^n\ev_j^k  \w_j=\ev_1^k \left(\w_1+\sum_{j=2}^n\left(\frac{\ev_j}{\ev_1}\right)^k  \w_j\right).
$$
Let us denote $\bfxi_k:=\sum_{j=2}^n\left(\frac{\ev_j}{\ev_1}\right)^k  \w_j$, i.e. $\x_k=\ev_1^k (\w_1+\bfxi_k)$. Then  
$$
\dfrac{\x_k}{\|\x_k\|_{2}}=\dfrac{\ev_1^k (\w_1+\bfxi_k)}{\ev_1^k \|\w_1+\bfxi_k\|_{2}}=\dfrac{\w_1+\bfxi_k}{\|\w_1+\bfxi_k\|_{2}},
$$
where the first equality follows from $\ev_1\succ 0$ (since $\ev_1$ is at most finite and $\ev_1[0]=1$). 
Therefore,
$$
\dfrac{\x_k}{\|\x_k\|_{2}}=\left(\w_1+\bfxi_k\right)\dfrac{1}{\|\w_1+\bfxi_k\|_{2}}\wkto \w_1\dfrac{1}{\|\w_1\|_{2}}
$$
due to Lemma \ref{lemma::xik}, and continuity of weak convergence with respect to the addition and multiplication, Lemma \ref{lemma::continuityofwc}.
\end{proof}

\begin{theorem}[Power iteration with maximum norm]\label{Thm::pimax}
Under the conditions of Theorem \ref{Thm::pil2} if there
exists a unique coordinate $i_0$ such that 
$$
|\w^{i_0}_1[0]|=\max_i |\w^{i}_1[0]|,
$$
 then for $\x_k:={A^k \x}$ holds
$$
\dfrac{\x_k}{\|\x_k\|_{\max}} \wkto \dfrac{\w_1}{\|\w_1\|_{\max}}.
$$
\end{theorem}
The proof is completely analogues to the proof of Theorem \ref{Thm::pil2} due to 2* of Lemma \ref{lemma::xik}.

\begin{remark}
Observe that the condition \eqref{eq::xaswj} holds for all $\x=(\x^1,\dots,\x^n)^\mathrm{T}\in~ \mathcal C^n$ with at most finite entries, since
$$
\mathbf w_i=\left(\sum_{k=1}^n \x^k \overline {\mathbf v}_i^k\right)\mathbf v_i,
$$
where $\mathbf v_i$ is the eigenvector with $\|\mathbf v_i\|_2=1$, i.e. all the entries on the right-hand side are at most finite. 
\end{remark}

    \begin{proposition}[Power iteration for eigenvalues]\label{prop::eiPI}
Let $A\in \C^{n\times n}$ be a diagonalizable matrix with at most finite entries, let $\mathbf v$ be its eigenvector with at most finite entries, corresponding to the eigenvalue $\ev$ and $\|\mathbf v\|_2[0]>0$. Let $(\mathbf u_k)_k\subset\C^n$ be such that each $(\mathbf u_k)$  has at most finite entries. If  $\mathbf u_k\wkto \mathbf v$ pointwise and is regular pointwise, then 
		\begin{align*}
			\frac{\mathbf u_k^*A \mathbf u_k}{\|\mathbf u_k\|^2_{2}}\wkto \ev,
		\end{align*}
where $\mathbf {u}_k^*=(\overline {\mathbf u}_k^1,\dots \overline{\mathbf u}_k^2, \dots,\overline{\mathbf u}_k^n)$.
	\end{proposition}
	\begin{proof}
        Firstly, the sequences $(\|\mathbf u_k\|_2^2)_k\subset \R$ and $(\mathbf u_k^*A \mathbf u_k)_k\subset \C$ are at most finite by direct calculations and regular by Lemma \ref{lemma:reg}. Further, $\|\mathbf u_k\|^2\to \|\mathbf v\|^2$  and  ${\mathbf u_k^*A \mathbf u_k}\wkto \mathbf v^*A \mathbf v$ by Lemma \ref{lemma::continuityofwc}. Then by Theorem \ref{thm::regularityOfinverse} we get that
$$
\dfrac{1}{\|\mathbf u_k\|^2}\to \dfrac{1}{\|\mathbf v\|^2},
$$
from where, again by Lemma \ref{lemma::continuityofwc}, follows
$$
\dfrac{\mathbf u_k^*A \mathbf u_k}{\|\mathbf u_k\|^2}\to \dfrac{\mathbf v^*A \mathbf v}{\|\mathbf v\|^2}.
$$
Finally,
$$
\dfrac{\mathbf v^*A \mathbf v}{\|\mathbf v\|^2}=\dfrac{\mathbf v^*\ev \mathbf v}{\|\mathbf v\|^2}=\ev.
$$
\end{proof}

\begin{remark}
By the proofs of Theorem \ref{Thm::pil2} and Theorem \ref{Thm::pimax} the obtained sequences of vectors $(\x_k)$ are regular. Therefore, Theorem \ref{prop::eiPI} allows us to calculate eigenvalues for corresponding matrices using the power iteration algorithm.
\end{remark}

\section{Applications of power iteration algorithm}

In Theorem \ref{Thm::pil2} and Theorem \ref{Thm::pimax} we assumed that for the eigenvalues $\ev_1, \dots, \ev_n\in \C$ of matrix $A\in \C^{n\times n}$ the following conditions are fulfilled:
\begin{enumerate}[label=$(\mathrm{C}\arabic*)$]
\item
all $\ev_i$ are at most finite,
\item
$$
\ev_1[0]>\left|\ev_2[0]\right|\ge\dots\ge|\ev_n[0]|,
$$
\item
$\ev_1[0]=1$,
\end{enumerate}
 In Proposition \ref{GCT} we show that the first condition is fulfilled for any  matrix with at most finite entries. Further, in Proposition \ref{prop:pisigma} we show, that for a matrix $A\in \C^{n\times n}$ with at most finite entries complex parts of eigenvalues
(i.e. $\ev[0]$) can be obtained as eigenvalues of the matrix  $A[0]=(a_{ij}[0])$ for $A=(a_{ij})$, i.e.  the condition (C2) can be checked without knowledge of eigenvalues of $A$. Moreover, then
the matrix $A/\ev_1[0]$  can be calculated, and will satisfy all conditions (C1)--(C3), see Corollary \ref{thm::generalPM}.

Further in this Section, we discuss, how the power iteration algorithm can be applied for some matrices with infinitely large entries or whose all entries are infinitely small, see Theorem \ref{thm::generalPM} and we show, that the power iteration algorithm, applied to the matrices over the field of Puiseux series leads to the result be an element of the same field, Corollary \ref{cor:Puiseux}.

\begin{proposition}[Gershgorin Circle Theorem]\label{GCT}
Let $A=(a_{ij})\in \C^{n\times n}$. The unity of the circles 
$$
K:= \bigcup_{i=1}^n  \{ z \in \C\:\mid\; | z - a_{ii}| \preceq \sum_{\substack{k=1, \\k \neq i}}^n |a_{ik}|\;\}
$$ 
contains all eigenvalues of $A$.
\end{proposition}
\begin{proof}
Let $\mathbf v$ be an eigenvector of $A$, $\|\mathbf v\|_{\max}=1$, and $\ev$ be the corresponding eigenvalue. Let $j$ be such that $|\mathbf v^j| = \|\mathbf v\|_{\max} = 1$.

 From $A\mathbf v= \ev \mathbf v$ follows:
    $$(a_{jj} - \ev)\mathbf v^j + \sum_{\substack{k=1,\\ k \neq j}}^n a_{jk}\mathbf v^k = 0.$$

Hence, 
    $$
\abs{(a_{jj} - \ev)}=\abs{(a_{jj} - \ev)}|\mathbf v^j|=\abs{(a_{jj} - \ev)\mathbf v^j}\preceq \big|\sum_{\substack{k=1,\\ k \neq j}}^n a_{jk}\mathbf v^k\big|\preceq \sum_{\substack{k=1,\\ k \neq j}}^n \abs{a_{jk}},
$$
where the last inequality is true due to the triangle inequality and since
$$
|\mathbf v^k|\preceq \|\mathbf v\|_{\max} = 1.
$$
\end{proof}

\bigskip
Next, we show in Proposition \ref{prop:pisigma}, that in the case of at most finite matrix, the set of complex parts of its eigenvalues coincides with the set of eigenvalues of the complex part of the matrix. In order to do this we need the following notation: for any $a\in \R$ (resp. $\C$) we denote by $\pr(a)$ the \emph{real part} (resp.  \emph{complex part}) of $a$, i.e. 
$$
\pi:\R \mbox{ (resp. }\C\mbox{)}\to \mathbb R  \mbox{ (resp. }\mathbb C\mbox{)}: a\to a[0].
$$

The following immediately follows:
\begin{lemma}\label{lemma::linearityofpr}
For all at most finite $a,b\in \R$ (resp.  $\C$) the map $\pr$ is linear, i.e.
\begin{itemize}
\item
$\pr(a+b)=\pr(a) + \pr(b),$
\item
$\pr(ab)=\pr(a)\cdot \pr(b).$
\end{itemize}
\end{lemma}

\begin{proposition}\label{prop:pisigma}
Let $A \in \mathcal{C}^ {n \times n}$ be a matrix with at most finite entries. Then
$$\pi(\sigma (A))  = \sigma (\pi(A)),$$
where $\sigma$ denotes the spectrum of a matrix.
\end{proposition}

\begin{proof}
Firstly, we introduce a mapping $\pi_P$, corresponding to $\pi$, on polynomials over $\C$ with at most finite coefficients 
$$
\pi_P: \C[x]\to \Bbb C[x]: \sum_{i=0}^n a_ix^i\to \sum_{i=0}^n \pi(a_i)x^i.
$$

Secondly, we note, that due to Gershgorin Circle Theorem Proposition~\ref{GCT}, every eigenvalue $\ev$ of $A$ is at most finite.

Let us denote by $\chi$ the characteristic polynomial of $A$. Then, by the linearity of $\pi$, Lemma \ref{lemma::linearityofpr},  $\pi (A)$ has the characteristic polynomial  $\pi_P(\chi)$. 

``$ \pi(\sigma (A))  \subseteq \sigma (\pi (A))$''
\newline
Let 
$$
\chi=\sum_{i=0}^n b_ix^i.
$$
Let $\ev\in \sigma(A) $. Then $\chi(\ev)=0$ and we get 
$$
(\pi_P(\chi))(\pi (\ev))=\sum_{i=1}^n \pi (b_i) (\pi (\ev))^i=\pi \left(\sum_{i=1}^n b_i \ev^i\right)=\pi (\chi(\ev))=0,
$$
where the second equality is true due to linearity of $\pi$, Lemma \ref{lemma::linearityofpr}.
Hence, $\pi (\ev) \in \sigma (\pi (A))$.

`` $\pi(\sigma (A))  \supseteq \sigma (\pi(A))$''. Since $\C$ is algebraically closed, we can write
$$
\chi=\prod_{i=1}^n (\ev_i-x),
$$
where $\ev_i$ are eigenvalues of $A$. Then again due to the linearity of $\pi$, Lemma~\ref{lemma::linearityofpr}, we obtain
$$
\pi_P (\chi)=\prod_{i=1}^n (\pi (\ev_i)-x),
$$
is a linear factor representation of $\pi_P (\chi)$ which is known to be unique.  Therefore for any $\mu\in \sigma(\pi(A))$, i.e $(\pi_P (\chi))(\mu)=0$ we  conclude that there exists $\pi (\ev_i)=\mu$, i.e. $\mu \in \pi (\sigma (A))$.

\end{proof}

\begin{corollary}
Let $A\in \mathcal C^{n\times n}$ be an at most finite diagonalizable matrix. If all the eigenvalues of $\pi(A)$ have different absolute values, then all the eigenvalues of $A$ can be found by the power iteration method. 
\end{corollary}
\begin{proof}
This can be done, for example via subspace iteration. Furthermore, it is possible to use the inverse vector iteration. All of these methods are based directly on the power iteration.
\end{proof}

Let $B\in \mathbb C^{n\times n}$. We say that $B$ has a \emph{dominated eigenvalue} $\mu_1$ if
$$
|\mu_1|>\left| \mu_2\right|\ge\dots\ge|\mu_n|,
$$ 
where $\mu_1, \mu_2,\dots \mu_n$ are eigenvalues of $B$.

\begin{theorem}\label{thm::generalPM}
Let $A\in \C^{n\times n}$ be diagonalizable. Let $A' =d^{-q_0} A$, where $q_0=\min_{1\le i,j\le n}\lambda(a_{ij})$. If the matrix $\pi(A')\in \Bbb C^{n\times n}$ has a dominated eigenvalue, then the largest in modulus eigenvalue of $A$ and a corresponding eigenvector can be found by power iteration method.
\end{theorem}

\begin{proof}
Due to the definition of $A'$ all its entries are at most finite. Assume the eigenvalues of the matrix $\pi(A')$ to be
$$
|\mu_1|>\left| \mu_2\right|\ge\dots\ge|\mu_n|.
$$
Due to Proposition  \ref{prop:pisigma} we conclude, that the eigenvalues of $A'$ satisfy 
$$
|\ev'_1[0]|>\left|\ev'_2[0]\right|\ge\dots\ge|\ev'_n[0]|.
$$
and $\mu_1=\ev'_1[0]$.
 Therefore, one can apply Theorem \ref{Thm::pil2} or Theorem \ref{Thm::pimax} and Proposition \ref{prop::eiPI} afterward to the matrix $A'/\mu_1$ to obtain its largest eigenvalue $\ev$. Then the largest in modulus eigenvalue of $A$ is
$$
\ev_1=\ev \mu_1 d^{-q_0}.
$$
\end{proof}

\begin{remark}
We do not discuss in this note, which matrices over $\R$ or $\C$ are diagonalizable. For example, symmetric matrices and matrices of symmetric operators on $\R^n$ are diagonalizable over $\R$. Further, the companion matrices of polynomials with pairwise different roots in $\R$ or $\C$ are diagonalizable, see Section 6.
\end{remark}

Now we will formulate Corollary \ref{cor:Puiseux}, which allows us to apply the power iteration algorithm to the field of Puiseux series. Let us recollect, that for $\Bbb K$ being the field of real or complex numbers, the field of \emph{Puiseux series} is a set of formal power series:
\begin{equation*}
\mathbb K\{\{t\}\} = \left\{ \sum_{k=k_0}^{\infty} s_k t^\frac{k}{m} \,:\, s_k \in \mathbb K \text{ and } k_0 \in \mathbb{Z}, m \in \mathbb{N} \right\}
\end{equation*}
with sum and multiplication defined naturally as for formal series.

The field of Puiseux series $\mathbb R\{\{t\}\} $ is a subfield of the Levi-Civita field $\R$ due to Proposition \ref{prop::powerseriesrepresentation}. Moreover, $\mathbb R\{\{t\}\} $ is a real-closed field itself, while $\mathbb C\{\{t\}\} $ is algebraically closed,   see e.g. \cite{Hall}. This field is much more used in algebraic geometry and computer science, since Cauchy completeness with respect to strong convergence, which the Levi-Civita field possesses, is not usually needed there.  Note that $\mathbb R\{\{t\}\}\subset \R$ and $\mathbb C\{\{t\}\}\subset \C$ are subfields, therefore for the elements $a\in \mathbb R\{\{t\}\} $ (resp. $ \mathbb C\{\{t\}\} $) apply the same notations as for $a\in \mathcal R$ (resp. $C$), in particular if $a=\sum_{k=k_0}^{\infty} s_k t^\frac{k}{m}$, then
$$
\lambda(a):=\dfrac{k_0}{m},
$$
and if in addition $k_0\ge 0$, then 
$$
\pi(a)=
\begin{cases}
0, k_0>0,\\
k_0, \mbox{ otherwise}.
\end{cases}
$$
 Then the following holds for matrices over the fields of Puiseux series:
\begin{corollary}\label{cor:Puiseux}
Let a matrix $A\in \Bbb K\{\{t\}\}^{n\times n}$ be diagonalizable over $\mathbb K\{\{t\}\}$, $\Bbb K=\mathbb R$ or $\mathbb C$.  Let $A' =t^{-q_0} A$, where $q_0=\min_{1\le i,j\le n}\lambda(a_{ij})$. If the matrix $\pi(A')\in \Bbb C^{n\times n}$ has a dominated eigenvalue, then the largest in modulus eigenvalue of $A$ and corresponding eigenvector can be found by power iteration method.
\end{corollary}

\begin{proof}
Is an immediate consequence of Theorem \ref{thm::generalPM}.
\end{proof}

\section{Complexity and experimental results}
We now have a look at the time complexity of the power iteration algorithm and the cost of arithmetic operations of polynomials. We investigate the complexity analysis first for companion matrices of polynomials $P$ over $\Bbb K\{\{t\}\}$, whose coefficients are polynomials in $t$.
Looking at our matrix-vector multiplication with $A \in \Bbb K\{\{t\}\}^{n\times n}$  and $A$ being a companion matrix, we have $n$ multiplications. Using the FFT Polynomial Multiplication we can assume $\mathcal{O} (n d_t \log(d_t))$ complexity, where $d_t$ is the maximum degree regarding $t$ over all matrix entries. In our normalization step, we have to add and square our polynomials, resulting in a complexity of $\mathcal{O} (d_t \log(d_t)n)$. Root taking and inversion can be neglected with a cost of $\mathcal{O} (d_t \log(d_t))$ \cite{haoze}. Concluding our analysis, we have a time complexity of $\mathcal{O} (d_t \log(d_t)n)$. For arbitrary matrices, we have $n^2$ matrix multiplications resulting in a complexity of $\mathcal{O} (n^2 d_t \log(d_t))$. 
\medskip
\newline
Now we will present some experimental results from our implementation to calculate the root of a polynomial over the Puiseux field. We remind the following definition:
\begin{definition}
The companion matrix of the polynomial $P\in \mathcal K[x]$, where $\mathcal K$ is an  arbitrary field and
$$
 {P(x)=a_{0}+a_{1}x+\cdots +a_{n-1}x^{n-1}+x^{n}},\; a_i\in \mathcal K,
$$
is the square matrix defined as
$$
 {\displaystyle C(P)=\begin{bmatrix}0&0&\dots &0&-a_{0}\\1&0&\dots &0&-a_{1}\\0&1&\dots &0&-a_{2}\\\vdots &\vdots &\ddots &\vdots &\vdots \\0&0&\dots &1&-a_{n-1}\end{bmatrix}}.
$$
\end{definition}
Note that the characteristic polynomial as well as the minimal polynomial of ${\displaystyle C(P)}$ are equal to $P$, and, hence, the eigenvalues of $C(P)$ are exactly the roots of $P$.

 Its recreation can also be found in the GitHub repository \cite{ebker}.
We use a polynomial of degree $21$ over $\Bbb R\{\{t\}\}$ from linear factor representation as an example.
It's largest root is: $$100 + t + 2t^2 + 3t^3 + 4t^4 + 5t^5 + 6t^6 + 7t^7 + 8t^8 + 9t^9.$$
We use its companion matrix to test our algorithm. The remaining roots are of the form $(2n+\frac{n}{20}t), n \in \{1,...,20\}$. The convergence is visualized in Figure \ref{figure_conv}:
 \begin{figure}[H]
\centering
\includegraphics[width=1\linewidth]{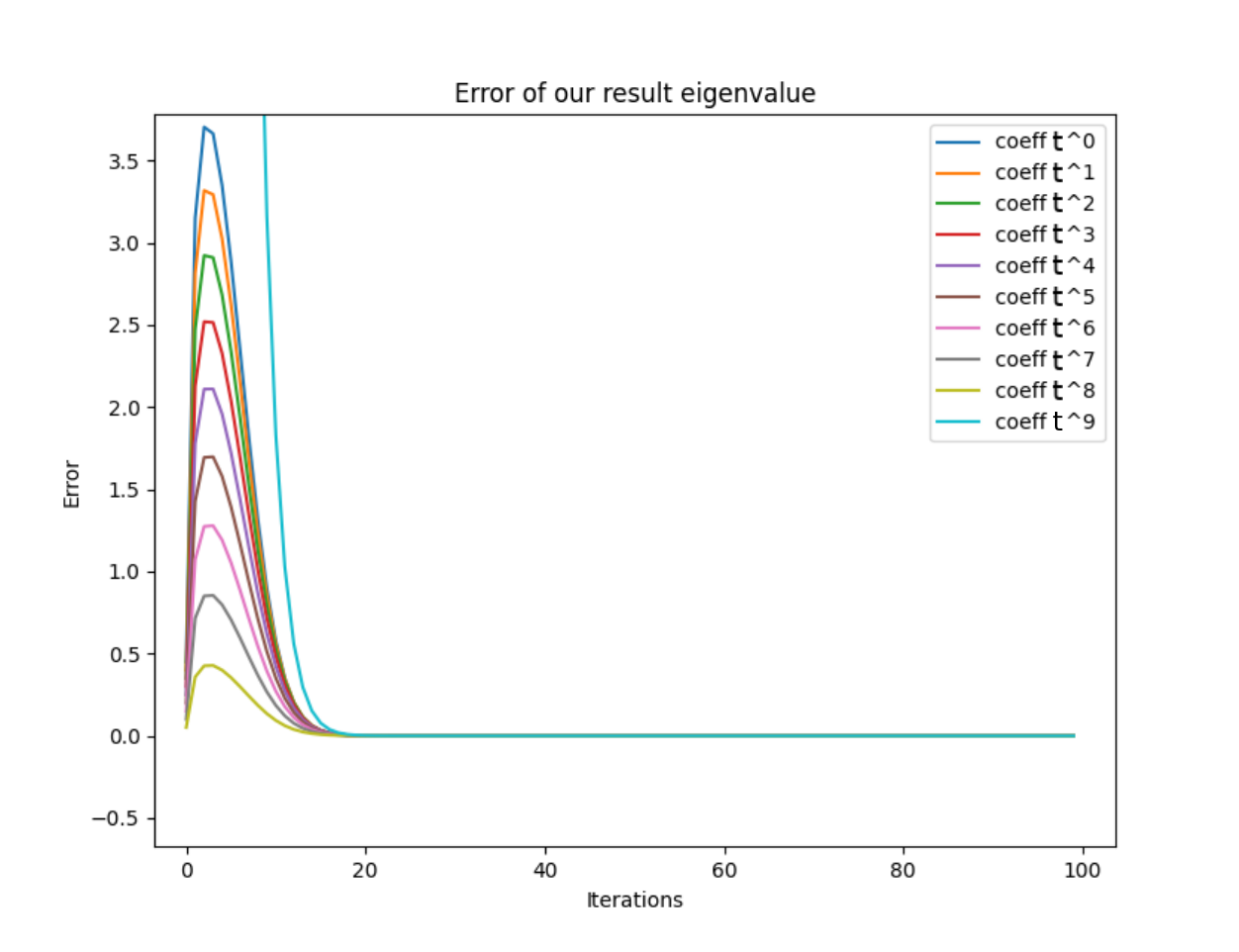}
\caption{} \label{figure_conv}
\end{figure}
We obtain the following error table for the first five components:
   \begin{table}[!ht]
           \footnotesize

    \centering
    \begin{tabular}{|l|l|l|l|l|l|}
    \hline
         Step & $t^0$ & $t^1 $& $t^2$ & $t^3$ & $t^4$ \\ \hline
        0 & 3.79432e+02 & 5.05098e-02 & 1.00464e-01 & 1.50350e-01 & 2.00276e-01 \\ \hline
        10 & 1.82443e+00 & 9.32998e-02 & 1.84316e-01 & 2.70750e-01 & 3.50363e-01 \\ \hline
        20 & 1.85048e-03 & 2.74954e-04 & 5.28935e-04 & 7.41959e-04 & 8.96181e-04 \\ \hline
        30 & 4.21493e-07 & 1.04826e-07 & 1.96292e-07 & 2.62170e-07 & 2.92486e-07 \\ \hline
        40 & 5.26370e-11 & 1.81429e-11 & 3.27154e-11 & 4.09770e-11 & 4.21161e-11 \\ \hline
        50 & 1.08002e-12 & 9.39249e-14 & 1.73417e-13 & 7.90479e-14 & 4.52971e-14 \\ \hline
        60 & 7.13385e-12 & 6.54699e-13 & 8.59979e-13 & 1.40510e-12 & 2.35723e-12 \\ \hline
        70 & 5.40012e-13 & 9.41469e-14 & 1.34115e-13 & 1.96732e-13 & 4.40981e-13 \\ \hline
        80 & 1.47793e-12 & 1.30340e-13 & 3.46390e-14 & 2.01172e-13 & 1.19016e-13 \\ \hline
        90 & 1.30740e-12 & 3.49720e-14 & 4.70735e-13 & 7.91811e-13 & 3.06422e-13 \\ \hline
        100 & 3.65219e-12 & 6.88560e-13 & 6.76792e-13 & 6.43041e-13 & 1.46061e-12 \\ \hline
    \end{tabular}
\end{table}

The last five components have similar errors.

The code took roughly 135 seconds for 100 iterations on a Lenovo p50. The performance of the code can be improved heavily with a modified version of the algorithm. 

\begin{remark}
    We have also implemented the methods of inverse vector iteration, subspace iteration and the QR algorithm. They are useful to calculate the remaining vectors or to calculate all the vectors simultaneously and efficiently.
\end{remark}
\bibliographystyle{plain}

\end{document}